\documentclass{amsart}
\usepackage{amssymb, amsmath,amsmath,latexsym,amssymb,amsfonts,amsbsy, amsthm,mathtools,graphicx,CJKutf8,CJKnumb,CJKulem,color}
\usepackage{graphicx}
\usepackage{amssymb}
\usepackage{graphicx}
\usepackage{graphicx,epsfig}
\usepackage{amsmath}
\usepackage{mathrsfs}
\usepackage{amssymb}

\usepackage{amssymb,mathrsfs,amsfonts,amsmath,amsbsy,tikz}\usepackage{fontenc}\usepackage{textcomp}

\numberwithin{equation}{section}

\allowdisplaybreaks

\def\R{\mathbf R}
\def\Z{\mathbf Z}

\def\C{\mathbb{C}}
\def\R{\mathbb{R}}
\def\Z{\mathbb{Z}}

\def\I{\mathcal{I}}

\newcommand{\beq}{\begin{equation}}
\newcommand{\eeq}{\end{equation}}
\newcommand{\ben}{\begin{eqnarray}}
\newcommand{\een}{\end{eqnarray}}
\newcommand{\beno}{\begin{eqnarray*}}
\newcommand{\eeno}{\end{eqnarray*}}

\let\tht=\theta

\let\ld=\lambda

\let\ga=\gamma
\let\De=\Delta
\let\nb=\nabla

\let\tht=\theta
\let\Tht=\Theta
\let\af=\alpha
\let\bt=\beta

\let\Ld=\Lambda

\newcommand{\ba}{\begin{array}}
\newcommand{\ea}{\end{array}}
\newcommand{\be}{\begin{equation}}
\newcommand{\ee}{\end{equation}}
\newcommand{\ban}{\begin{eqnarray*}}
\newcommand{\ean}{\end{eqnarray*}}

\numberwithin{equation}{section}

\newtheorem{theorem}{Theorem}[section]

\newtheorem{lemma}[theorem]{Lemma}

\newtheorem{remark}[theorem]{Remark}

\begin{document}
\begin{CJK*}{UTF8}{gkai}
\title[Friedrichs extensions] {Friedrichs extensions  for a class of singular discrete
 linear Hamiltonian systems}

\author{Guojing Ren}
\address{School of statistics and Mathematics,
 Shandong  University of Finance and Economics, Jinan, Shandong 250014, P. R.
China}
\email{gjren@sdufe.edu.cn}

\author{Guixin Xu}
\address{School of Mathematics and Statistics, Beijing Technology and Business University, Beijing, P. R. China}
\email{guixinxu$_-$ds@163.com}

\maketitle

\begin{abstract}
This paper is concerned with the characterizations of the Friedrichs extension
for a class of singular discrete linear Hamiltonian systems.
The existence of recessive solutions and the existence of the Friedrichs extension are proved
under some conditions.
The self-adjoint boundary conditions are obtained
by applying the  recessive solutions  and then
the  characterization of the Friedrichs extension  is obtained in
terms of boundary conditions via linear independently recessive solutions.
\end{abstract}

{\bf \it Keywords}:\ Discrete  Hamiltonian system;  Friedrichs extension; Disconjugacy;  Recessive solution.

{2020 {\bf \it Mathematics Subject Classification}}:  34B20, 39A70.

\section{Introduction}

In this paper, we consider the  following singular discrete linear Hamiltonian system
\begin{eqnarray}\label{HS1}
\left\{\begin{array}{l}
         \De u(t)=A(t)u(t+1)+B(t)v(t),\\
         \De v(t)=(C(t)-\ld W(t))u(t+1)-A^*(t)v(t),\quad t\in \I,
       \end{array}\right.
\end{eqnarray}
where  $\De$ is the forward difference operator, i.e., $\De u(t)=u(t+1)-u(t)$;
$A(t)$ is $n\times n$ complex matrix, $B(t)$, $C(t)$ and $W(t)$ are $n\times n$ Hermitian  matrices satisfying
\begin{equation*}
 B(t)\ge 0,\quad W(t)\ge 0;
\end{equation*}
$\ld\in \C$ is a spectral parameter; and  $\I:=[0,\infty)\cap\Z=\{t\}_{t=0}^{\infty}$.
To ensure the existence and uniqueness of the solution of any
initial value problem for $(1.1)$, we always assume that
\begin{equation*}
I_n-A(t) \;{\rm is\; invertible\;on \;}\I,
\end{equation*}
where $I_n$ is the $n\times n$ unit matrix.

Self-adjoint extension problems are most fundamental in
the study of spectral problems for  differential operators.
There are two important theories available.
One is the Weyl-Titchmarsh theory, which was started with Weyl¡¯s work  in 1910  and
it has been generalized to linear Hamiltonian differential systems
(cf. \cite{Atkinson1964,HintonShaw1981,Naimark1968,Weidmann1980} and their references).
The other is the GKN theory, which was established by Glazman, Krein, and Naimark in 1950's .
Based on these two theories, the self-adjoint extensions for  linear  differential
operators  have been widely studied \cite{SunH2010,SunH2011,SunJ1986,Wang2009}.

Among all the self-adjoint extensions, there is a particular one which preserves
the same lower bound when the associated minimal operator is bounded from below.
This self-adjoint extension is known as the Friedrichs extension,
which was initially constructed by K. Friedrichs for a densely defined operator in 1930's.
His work has been widely developed by many authors to singular  differential operators using various approaches
(cf. \cite{Kalf1972,Kalf1978,Marletta-Zettl2000,Moller-Zettl1995,Niessen-Zettl1990,Niessen-Zettl1992}).
By using  the existence of the  principal solutions of linear Hamiltonian differential systems \cite{Reid1980},
a characterization of  the Friedrichs extension of a class of singular Hamiltonian differential operators
was given in terms of principal solutions \cite{Zheng2018}.
Recently, the  Friedrichs extension of a class of singular  differential systems including non-symmetric cases
was characterized \cite{YangSun2021}.

Difference equations are usually regarded as the discretization of the corresponding differential equations.
It has been found that the  maximal operator corresponding to a formally self-adjoint difference equation
is multi-valued, and the corresponding minimal operator is non-densely defined in general case \cite{Ren2014AML, Shi2011}.
Not only that,  this case can  happen for general linear Hamilton differential systems \cite{Lesch2003}.
Therefore, the classical spectral theory for symmetric operators, i.e.,
densely defined and Hermitian single-valued operators, are not available  in the studying
of the spectral properties of  differential and difference equations in general case.

Due to the above reasons, some researchers began to focus on extending the  theory of linear operators to
linear  non-densely defined or multi-valued operators (which are called linear relations or linear subspaces),
and many good results have been obtained (cf. \cite{Cross1998,Hassi2017,Shi2017} and their references).
E. A. Coddington successfully extended the von Neumann self-adjoint
extension theory for symmetric operators to Hermitian subspaces \cite{Coddington1973}.
Y. Shi extended the classical GKN theory for symmetric operators to Hermitian subspaces \cite{Shi2012}.
Based on the above, a complete characterization of all the self-adjoint extensions for a class of
discrete linear Hamiltonian systems are obtained in
terms of boundary conditions via linear independent square summable
solutions \cite{Ren2014Laa}.

The spectral properties of  discrete linear Hamiltonian  systems have been widely discussed
 (cf. \cite{Ahlbrandt-Peterson1996,Bohner1996,Erbe-Yan1992,Ren2011,Shi2006} and their references).
 It is worth mentioning  that
M. Bohner  and his coauthors discussed the disconjugacy  of discrete linear Hamiltonian  systems
and  proved the Reid Roundabout Theorem \cite[Theorem 2]{Bohner1996}.
In \cite{Ahlbrandt-Peterson1996}, some properties of recessive solutions of (1.1) were discussed.
The Friedrichs extension of  semi-bounded  second-order difference operators was discussed in \cite{Benammar-Evans1994}.
To the best of our knowlege, the existence of the recessive solutions of (1.1) has  not be given and the
Friedrichs extension of the minimal subspace corresponding to (1.1) has not been established.

The rest of this paper is organized as follows.
In Section 2, the basic concepts and useful results on the linear relations and   the sesquilinear forms
are recalled. In addition, several important difference equations related with (1.1) are listed at the end of this section.
In Section 3, the existence of the recessive solutions of (1.1) and  the existence  of the Friedrichs extension
of the minimal subspace generated by (1.1) are proved under some conditions, respectively.
In addition, a characterization of the  recessive solutions of (1.1) is obtained.
In Section 4, a characterization of matrix $\Tht$ is obtained by using the  recessive solutions
and then a characterization of the  Friedrichs extension is established in terms of boundary condition via the recessive solutions.

\section{Preliminaries}

In this section, some fundamental and useful results about  linear relations and   sesquilinear forms
are recalled.
At the end of this section, several important difference equations related to (1.1) are listed.

\subsection{Linear relations}

Let $\mathcal{X}$  be a complex Hilbert  space  with inner product $\langle\cdot, \cdot\rangle$.
Let $T$ and $S$ be  two linear relations (briefly, relations or named linear subspaces) in $\mathcal{X}^2:=\mathcal{X}\times \mathcal{X}$.
The domain and the range of $T$  are  defined as follows
\begin{eqnarray*}
&&\mathcal{D}(T):=\{x\in \mathcal{X}:\; (x,x')\in T \},\\
&&\mathcal{R}(T):= \{x'\in \mathcal{X}: \;(x,x')\in T \}.
\end{eqnarray*}
The adjoint of $T$ and the sum of two  relations are defined as
\begin{eqnarray*}
&&T^*:=\{(y,y')\in \mathcal{X}^2: \;\langle x,y'\rangle=\langle x',y\rangle\; {\rm  for\; all}\; (x,x')\in T\},\\
&&T^{-1}:=\{(x',x):  (x,x')\in T\},\\
&&T+S:=\{(x,x'+x''): \; (x,x')\in T, \;  (x,x'')\in S\}.
\end{eqnarray*}
It has been shown that $T$ is densely defined if and only if  $T^*$
is single-valued \cite[Theorem 3.1]{Ren2014AML}.
$T$ is said to be {\it Hermitian}  if $T\subset T^*$.
 $T$ is said to be {\it symmetric}   if $T\subset T^*$ and $\mathcal{D}(T)$ is dense in $\mathcal{X}$.
 $T$ is said to be {\it self-adjoint}  if $T=T^*$.

Denote
\begin{equation*}
  T(x)=\{x'\in \mathcal{X}:\; (x, x')\in T\}.
\end{equation*}
It is clear that $T(0)=\{0\}$ if and only if there exists
a unique linear  operator $A_T$ from $\mathcal{D}(T)$ into $\mathcal{X}$
such that its graph is equal to $T$, i.e., $G(A_T)=T$.

Let $\ld\in \C$. The subspace $\mathcal{R}(T-\ld I))^\bot$ and the number
$d_{\ld}(T) := \dim \mathcal{R}(T-\ld I)^\bot$ are called the {\it deficiency
space} and {\it deficiency index} of $T$ with $\ld$, respectively.
By $\overline{T}$ denote the closure of $T$ in $\mathcal{X}^2$.
It can be easily verified that $d_{\ld}(T)=d_{\ld}(\overline{T})$ for all $\ld\in \C$.

Let $T$ be a relation in $\mathcal{X}^2$. The set
  \begin{equation*}
  \Gamma(T):= \{\ld\in \C: \exists\, c(\ld) > 0 \;{\rm s.t.}\; \|x'-\ld x\|\ge c(\ld)\|x\|,
  \forall (x,x')\in T\}
\end{equation*}
is called the regularity domain of $T$.
It has been shown by  \cite[Theorem 2.3]{Shi2012} that the deficiency index $d_{\ld}(T)$
is constant in each connected subset of $\Gamma(T)$.
If $T$ is Hermitian, then $d_{\ld}(T)$ is constant in the upper and lower half-planes.
Denote
\begin{equation*}
  d_{\pm}(T) := d_{\pm i}(T)
\end{equation*}
 for an  Hermitian linear relation $T$,
and call $d_{\pm}(T)$ the positive and negative deficiency indices of  $T$, respectively.

\begin{lemma} \cite[Theorem 15]{Coddington1973}
Let $T$ be a closed Hermitian subspace in  $\mathcal{X}^2$. Then $T$ has self-adjoint  extensions if and only if
$d_+(T)=d_-(T)$.
\end{lemma}

Next, we introduce a form on $\mathcal{X}^2\times \mathcal{X}^2$ by
\begin{align*}
[(x,x'):(y,y')]:=\langle x',y\rangle-\langle x,y'\rangle,\quad
(x,x'),(y,y')\in \mathcal{X}^2.
\end{align*}
It can be easily verified that $[:]$ is a conjugate bilinear and
skew-Hermitian map from $\mathcal{X}^2\times \mathcal{X}^2$ into $\C$.
Let $T$ be a closed Hermitian subspace in $X^2$ and satisfy
$d_+(T)=d_-(T)=d$. A set $\{\bt_j\}_{j=1}^{d}$ is called a GKN-set
for the pairs of subspaces $\{T,T^*\}$ if it  satisfies
\begin{itemize}
\item [{\rm (1)}] $\bt_j\in T^*$, $1\le j\le d$;
\item [{\rm (2)}] $\bt_1,\bt_2,\ldots,\bt_d$ are linearly independent in
$T^*$ {\rm(modulo $T$)};
\item [{\rm (3)}] $[\bt_j:\bt_k]=0$, $1\le j,k\le d$.
\end{itemize}

\begin{lemma} \cite[Theorem 4.7]{Shi2012}
Let $T$ be a closed Hermitian subspace in $\mathcal{X}^2$ and satisfy
$d_+(T)=d_-(T)=d$. A subspace $T_1$ in $\mathcal{X}^2$ is a self-adjoint extension  of $T$ if and
only if there exists a GKN-set $\{\bt_j\}_{j=1}^d$ for $\{T,T^*\}$
such that $T_1$ is determined by
\begin{align*}
    T_1=\{\gamma\in T^*:\;[\gamma:\bt_j]=0,\,1\le j\le d\}.
\end{align*}
\end{lemma}

A  subspace $T$  is said to be bounded from below  if there exists a number $\gamma\in \R$ such that
\begin{equation}\label{bounded blow}
  \langle x',x\rangle\ge \gamma\|x\|^2,  \quad  \forall\; (x,x')\in T.
\end{equation}
The lower bound of $T$ is the largest number $\ga\in \R$ for which
(\ref{bounded blow}) holds.
It can be easily verified that $T$ is  Hermitian if it is bounded from below.

\begin{lemma}\label{lower bound def index equal}\cite[Proposition 1.4.6]{Behrndt2020}
Let $T$ be an Hermitian subspace and be bounded from below with lower bound $\gamma$. Then
$\C\setminus [\gamma,\infty) \subset \Gamma(T)$, and the deficiency index $d_{\ld}(T)$ is constant for all
$\ld\in \C\setminus [\gamma,\infty)$.
\end{lemma}

\subsection{ Sesquilinear forms}

Let $\mathcal{X}$ be a complex Hilbert space and let $\mathcal{D}$ be a
linear subspace (not necessarily closed or dense) of $\mathcal{X}$.
A sesquilinear form  ${\bf t}[\cdot,\cdot]$ is a complex-valued mapping
from $\mathcal{D}\times \mathcal{D}$ to $\C$, which satisfies
\begin{align*}
 & {\bf t}[x,ay+bz]=\bar{a}\;{\bf t}[x,y]+\bar{b}\;{\bf t}[x,z], \\
 & {\bf t}[ay+bz,x]=a\;{\bf t}[y,x]+b\;{\bf t}[z,x],
\end{align*}
for all $x,y,z\in \mathcal{D}$ and $a,b\in \C$. The set $\mathcal{D}({\bf t})=\mathcal{D}$ is called the domain of ${\bf t}$.
Define the quadratic form ${\bf t}[\cdot]$ associated with $\bf t[\cdot,\cdot]$ as ${\bf t}[x]:={\bf t}[x,x]$
with the same domain $\mathcal{D}$.
${\bf t}$  is said to be sectorial with vertex at the origin and
semiangle $\af$, $\af\in [0,\pi/2)$, if
\begin{equation*}
  {\bf t_{i}}[x]\le \tan\af {\bf t_{r}}[x],\quad x\in \mathcal{D}({\bf t})=\mathcal{D}({\bf t_{i}})=\mathcal{D}({\bf t_{r}}),
\end{equation*}
where ${\bf t_r}$ and ${\bf t_i}$ stand for the real and imaginary parts of ${\bf t}$.
${\bf t}$ is said to be Hermitian if ${\bf t}[x]$ is real for all $x\in \mathcal{D}({\bf t})$.
It is clear that ${\bf t}$ is sectorial  if it is Hermitian.

A  sequence $\{x_n\}\subset\mathcal{D}(\bf  t)$ is said to be ${\bf t}$-convergent to $x\in \mathcal{X}$,
in symbol
\begin{equation*}\label{def closable}
  x_n  \stackrel{\bf t}{\longrightarrow}\;x,
\end{equation*}
if ${x_n}\to x$ in $\mathcal{X}$ and ${\bf  t}[x_n-x_m]\to 0$ as $n,m\to \infty$.
An Hermitian  sesquilinear form ${\bf t}$ is said to be closed if $\{x_n\} \stackrel{\bf t}{\longrightarrow}\;x$ implies that
$x\in \mathcal{D}(\bf  t)$ and ${\bf t}[x_{n}-x]\to 0$;
and ${\bf t}$  is said to be closable if it has a closed extension.

The following is a direct consequence of \cite[Chapter VI, Theorem 1.17]{Kato1984}.

\begin{lemma}\label{kato lemma}
An Hermitian  sesquilinear form ${\bf t}$ is  closable if and only if  $x_n  \stackrel{\bf t}{\longrightarrow}\;0$
implies $ {\bf t}[x_n]\to 0$ as $n\to \infty$.
When this condition is satisfied, ${\bf t}$ in $ \mathcal{X}$ has
the closure (the smallest closed extension) ${\bf \bar{t}}$ which defined in the following way.
\begin{align*}
& \mathcal{D}({\bf \bar{t}}):= \left\{x\in \mathcal{X}: \exists\;\{x_n\}\subset \mathcal{D}({\bf t})\;
{\rm s.t.}\; x_n  \stackrel{\bf t}{\longrightarrow}\;x\right\},\\
&{\bf \bar{t}}[x_1,x_2]:=\lim_{n\to \infty}{\bf t}[x_{1n},x_{2n}],\quad \forall
\;x_{jn}\stackrel{\bf t}{\longrightarrow}\;x_j,\quad j=1,2.
\end{align*}
\end{lemma}

The following result is a  direct consequence  of \cite[Lemma 4.2]{Hassi2017}.

\begin{lemma}\label{Hermite relation to form}
Let $T$ be a Hermitian relation in $\mathcal{X}^2$. The form ${\bf t}_{T}$ generated  by $T$ as the following
\begin{equation*}
{\bf t}_{T}[x_1,x_2]:=\langle x_1',x_2\rangle,\quad \forall\;(x_j,x_j')\in T,\quad j=1,2,
\end{equation*}
with $ \mathcal{D}(\mathbf{t})= \mathcal{D}(T)$ is well-defined, Hermitian  and closable.
\end{lemma}

The following result is a direct consequence of \cite[Proposition 2.8, Theorem 4.3]{Hassi2017}.

\begin{lemma}\label{existence F extension}
Let ${\bf t}$ be a closed and  Hermitian form in $\mathcal{X}$ with lower bound $\gamma>0$.
Then there exists a unique self-adjoint linear relation  $T_{\bf t}$ with the same lower bound $\gamma>0$ such that
\begin{equation}\label{form and relation 1}
\mathcal{D}(T_{\bf t})\subset \mathcal{D}(\bf t)
\end{equation}
and  for every $(x,x')\in T_{\bf t}$ and $y\in \mathcal{D}(\bf t)$ one has
\begin{equation}\label{form and relation 2}
{\bf t}[x,y]=\langle x',y\rangle.
\end{equation}
Conversely, for every self-adjoint relation  $T_{\bf t}$ with the lower bound $\gamma>0$, there exists a unique
closed and  Hermitian form  ${\bf t}$ in $\mathcal{X}$ such that $(\ref{form and relation 1})$ and $(\ref{form and relation 2})$ are satisfied.
\end{lemma}

\subsection{Several difference equations related to (1.1) }

System (1.1) contains the following two important models. One is the
 formally self-adjoint scalar difference equation with complex coefficients
\begin{align}\label{2nth equation}
\sum_{j=0}^{n}(-1)^j\Delta^j(p_j\nb^jz(t))+i\sum_{k=1}^{n}[(-1)^{k+1}\Delta^k(q_kz(t))+q_k\nb^kz(t)]=\ld w(t)z(t),
\end{align}
where $p_j(t)$ and $q_k(t)$ are all real-valued, $p_n(t) \neq 0$ on $\I$, and $i=\sqrt{-1}$.
In fact, by letting
$u=(u_1,u_2,\ldots,u_{n})^T$ and $v=(v_1,v_2,\ldots,v_{n})^T$ with
\begin{eqnarray*}
         &&u_j(t)=\Delta^{j-1}z(t-j),\\
        && v_{j}(t)=\sum_{k=j}^{n}(-1)^{k+j}\Delta^{k-j}(p_k(t)\nabla^kz(t))+i\sum_{k=j}^{n}(-1)^{k+1}\Delta^{k-j}(q_k(t)z(t)),
     \end{eqnarray*}
for $j=1,2,\ldots,n$, we can convert (\ref{2nth equation})  into  (1.1) with
\begin{eqnarray*}
&&A(t)=\begin{pmatrix}0 & I_{n-1}\\ iq_{n}(t)/p_{n}(t)
&0\end{pmatrix},\quad
C(t)=\begin{pmatrix} p_0(t)+q_{n}(t)/p_{n}(t)& \alpha(t)
\\\alpha^*(t),
&\bt(t)\end{pmatrix},
\end{eqnarray*}
and
\begin{align*}
& B(t)= {\rm diag}\{0,\ldots,0,p_n^{-1}(t)\},\\
&\alpha(t)=i(q_{n-1}(t),q_{n-2}(t),\ldots,q_1(t))^T, \\
&\bt(t)={\rm diag}\{p_1(t),p_2(t),\ldots,p_{n-1}(t)\},\\
&W(t)= {\rm diag}\{w(t),0,\ldots,0\}.
\end{align*}
The other is the   second-order vector Sturm-Liouville difference equation
\begin{equation}\label{Sturm-Liouville equation}
  -\nabla (P(t)\Delta u(t))+Q(t)u(t)=\ld W(t)u(t),\quad t\in \I,
\end{equation}
where $P(t)$, $Q(t)$  and $W(t)$ are  $n\times n$ Hermitian matrices,
$W(t)\ge 0$ and $P(t)>0$. In fact, by letting $v(t)=P(t)\De u(t)$, we can convert (\ref{Sturm-Liouville equation})
into the following form
\begin{eqnarray*}
\left\{ \begin{array}{l}
\De u(t)=P^{-1}(t)v(t),\\
\De v(t)=(Q(t+1)-\ld W(t+1))u(t+1),
\end{array}\right. \quad t\in \I.
\end{eqnarray*}

The general singular discrete linear Hamiltonian system is in the form
\begin{align}\label{HS2}
    J\Delta y(t)=(\tilde{P}(t)+\ld \tilde{W}(t))R(y)(t),\quad t \in \I,
\end{align}
where $\tilde{W}(t)\ge 0$ and $\tilde{P}(t)$ are $2n\times 2n$ Hermitian matrices;
$J$ is the $2n\times 2n$ canonical symplectic matrix, that is,
\begin{equation*}
J=\left(\begin{array}{cc} 0&-I_n\\I_n&0\end{array}\right),
\end{equation*}
$y(t)=(u^T(t), v^T(t))^T$ with $u(t),v(t)\in \C^n$
and   $R(y)(t)$ is the partial right shift operator
\begin{equation*}
  R(y)(t)=\left(
               \begin{array}{c}
                 u(t+1)\\
                v(t) \\
               \end{array}
             \right).
\end{equation*}
System (1.1) is a special case of (\ref{HS2}) with
\begin{equation*}
  \tilde{W}(t)={\rm diag}\{W(t),0\}.
\end{equation*}

Since $I_n-A(t)$ is invertible, system (1.1) is equivalent to  the following system
\begin{eqnarray}\label{SS}
\left\{\begin{array}{l}
         u(t+1)=\tilde{A}(t)u(t)+\tilde{A}(t)B(t)v(t),\\
         v(t+1)=\tilde{C}(t,\ld)\tilde{A}(t)u(t)+D(t,\ld)v(t),\quad t\in \I,
       \end{array}\right.
\end{eqnarray}
where
\begin{eqnarray}\label{ABC}
\begin{array}{lll}
&\tilde{A}(t)=(I_n-A(t))^{-1},\quad \tilde{C}(t,\ld)=C(t)-\ld W(t),\\
&D(t,\ld)=\tilde{C}(t,\ld)\tilde{A}(t)B(t)+I_n-A^*(t).
\end{array}
\end{eqnarray}
For any  $\ld\in \R$, it can be easily verified that
\begin{eqnarray*}\label{symplitic condition}
\left(
  \begin{array}{cc}
  \tilde{A}(t)&\tilde{A}(t)B(t)\\
  \tilde{C}(t,\ld)\tilde{A}(t)&D(t,\ld) \\
    \end{array}
\right)^*J\left(
  \begin{array}{cc}
  \tilde{A}(t)&\tilde{A}(t)B(t)\\
  \tilde{C}(t,\ld)\tilde{A}(t)&D(t,\ld) \\
    \end{array}
\right)\equiv J.
\end{eqnarray*}
This implies  that (\ref{SS}), as well as (1.1), is  a symplectic system when $\ld\in \R$.

\section{ The existence of the recessive solutions and  the Friedrichs extension}

In the first subsection,  some variational properties of  the solutions of  $(1.1)$ are recalled,
and then the existence  of the recessive  solutions of $(\ref{HS1})$ is proved.
In the second subsection, by introducing  a  quadratic form associated with (1.1)
we show the existence  of the Friedrichs extension of the minimal subspace generated by (1.1).
A characterization of the recessive solutions is given in the last subsection.

\subsection{Existence  of the recessive  solutions}

Denote
\begin{align*}
l(\I):=\{y:\;y=\{y(t)\}_{t\in\I},\; y(t)\in \C^{2n}\}.
\end{align*}
For any two  $y_1,y_2\in l(\I)$, denote
\begin{equation*}
(y_1,y_2)(t):=y_2^*(t)Jy_1(t),
\end{equation*}
where  $J$ is the canonical symplectic matrix.
If the limit
\begin{equation*}
\lim_{t\to \infty}(y_1,y_2)(t)
\end{equation*}
exists and is finite, then
its limit is denoted by $(y_1,y_2)(\infty)$.
The natural difference operator corresponding to system $(1.1)$  is denoted by
\begin{align}\label{def L}
\mathcal{L}(y)(t)&:=\left(\ba{l}(I_n-A^*(t))v(t)-v(t+1)+C(t)u(t+1)\\
(I_n-A(t))u(t+1)-u(t)-B(t)v(t)\ea\right),\quad y\in l(\I).
\end{align}

\begin{lemma}\label{Liouville formulae}  \cite[Lemma 2.1]{Shi2006}
 For any two $y_1, y_2\in l(\I)$ and any $s,k\in\I$ with $s<k$,
\begin{align*}
    \sum_{t=s}^{k}[R(y_2)^*(t)\mathcal{L}(y_1)(t)- \mathcal{L}(y_2)^*(t)R(y_1)(t)]=(y_1,y_2)(k+1)-(y_1,y_2)(s).
\end{align*}
\end{lemma}

By small letters $y=(u;v)$
we denote a  vector-valued solution of (1.1) and by capital letter $Y=(U;V)$ denote a $2n\times n$ matrix-valued solution of (1.1).
Let $y_1$ and $y_2$ be two solutions of (1.1). It follows from
Lemma \ref{Liouville formulae} that
\begin{align*}
(y_1,y_2)(t)\equiv (y_1,y_2)(0)=c, \quad t\in \I.
\end{align*}
where $c$ is a constant. Further, we call $y$   a {\it prepared (or conjoined) solution} if
\begin{equation*}
(y,y)(t)\equiv 0.
\end{equation*}
Similarly,  $Y(t)=(U(t);V(t))$ is called  a {\it prepared (or conjoined) solution} if
\begin{equation*}
  Y^*(t)JY(t)\equiv0.
\end{equation*}
It follows from Lemma \ref{Liouville formulae} that $Y(t)$ is  a  prepared   solution of (1.1) if and only if
$U^*(t)V(t)$ is Hermitian for some $t\in \I$.

As it has been shown in Section 2.3, system (1.1) can be converted into (\ref{SS}), which is a symplectic system when $\ld\in \R$.
So,  the results in  \cite[Section 3]{Ahlbrandt-Peterson1996} can be applied to system (1.1) when $\ld\in\R$.
Note that   $\tilde{A}(t)=(I_n-A(t))^{-1}$ in (\ref{ABC}).

\begin{lemma}\label{construct conjoined solution} \cite[Theorems 3.32, 3.33]{Ahlbrandt-Peterson1996}
Let  $\ld\in \R$, and  $Y_0=(U_0;V_0)$ be a  prepared solution of $(1.1)$ such that
$U_0(t)$ is invertible for  $t\ge t_1\ge 0$. Then
\begin{align*}
  D(t)=U_0^{-1}(t+1)\tilde{A}(t)B(t)(U_0^{-1})^*(t),\quad t\ge t_1
\end{align*}
is Hermitian, and consequently
\begin{align*}
 S_0(t_1):=0,\quad  S_0(t):=\sum_{s=t_1}^{t-1}U_0^{-1}(s+1)\tilde{A}(s)B(s)(U_0^{-1})^*(s), \quad t\ge t_1+1
\end{align*}
is Hermitian for $t\ge t_1$.
Let $Y=(U;V)$ be another   solution of $(1.1)$. Then it can be expressed as
\begin{eqnarray}\label{def UV}
\left\{\begin{array}{ll}
         U(t)=U_0(t)(P+S_0(t)Q),&\\
         V(t)=V_0(t)(P+S_0(t)Q)+(U_0^{-1})^*(t)Q,\quad t\ge t_1,
       \end{array}    \right.
\end{eqnarray}
where
\begin{align}\label{def PQ}
 P=U_0^{-1}(t_1)U(t_1),\quad  Q=U_0^*(t)V(t)-V_0^*(t)U(t)
\end{align}
are  constant matrices.
Conversely, let $P$ and $Q$ be constant $n\times n$ matrices and $U(t)$ and $V(t)$ be defined by $(\ref{def UV})$. Then
$Y(t)=(U;V)$ is a solution of $(1.1)$ and therefore $(\ref{def PQ})$ holds.
Further, $Y=(U;V)$ is a prepared solution of $(1.1)$
if and only if $P^*Q$ is Hermitian.
\end{lemma}

Let $Y=(U;V)$ be  a  prepared solution  of (1.1).
It is  said to be a {\it  dominant solution}
provided that there exists an integer $t_1 \in \I$ such that $U(t)$ is invertible for $t\ge t_1$  and
\begin{align*}
 \sum_{s=t_1}^{\infty}U^{-1}(s+1)\tilde{A}(s)B(s)(U^{-1})^*(s)
\end{align*}
converges to a Hermitian matrix with finite entries.
We call  $Y(t)$  a {\it recessive (or principal) solution}
 provided that whenever  $Y_1=(U_1;V_1)$ is a solution of (1.1) satisfying
\begin{equation*}
 Y^*(t)JY_1(t)\equiv C, \quad t\in \I,
\end{equation*}
where $C$ is a non-singular matrix,  there exists  $t_1\in \I$ such that $U_1(t)$ is non-singular for $t\ge t_1$ and
\begin{equation*}
\lim_{t\to \infty} U_1^{-1}(t)U(t)=0.
\end{equation*}
In this case, $y_1(t), y_2(t),\ldots,y_n(t)$, which are the  column of $Y(t)$,  are called the recessive  solutions of (1.1).

\begin{lemma} \label{dominant to recessive solution} \cite[Theorems 3.50]{Ahlbrandt-Peterson1996}
Let  $\ld\in \R$, and  $Y=(U;V)$  be a dominant solution of $(1.1)$ with  $U(t)$  invertible for $t\ge t_1\ge 0$.
Define  $Y_0=(U_0;V_0)$  by
\begin{align*}
\left\{
 \begin{array}{ll}
   U_0(t)=U(t)S(t),\\
   V_0(t)=V(t)S(t)-(U^{-1})^*(t),& t\ge t_1,
 \end{array}\right.
\end{align*}
where
\begin{equation*}
 S(t)=\sum_{s=t}^{\infty}U^{-1}(s+1)\tilde{A}(s)B(s)(U^{-1})^*(s),\quad t\ge t_1,
\end{equation*}
Then $Y_0=(U_0;V_0)$ is a recessive solution of $(1.1)$.
\end{lemma}

System $(1.1)$ is said to have {\it the  unique two point property} on  $\I$ with respect to some $\ld=\ld_0$ provided
that whenever $y_1=(u_1;v_1)$ and $y_2=(u_2;v_2)$ are solutions of (1.1) with  $u_1(t_1)=u_2(t_1)$ and $u_1(t_2)=u_2(t_2)$,
where $0\le t_1<t_2<\infty$, it follows that $y_1(t)\equiv y_2(t)$ on $\I$.
It follows from \cite[Thoerem 3.55]{Ahlbrandt-Peterson1996} that there exists a unique matrix-valued solution of (1.1) satisfying the
boundary conditions
\begin{equation*}
  U(t_1)=I_n,\quad U(t_2)=0, \quad 0\le t_1<t_2,
\end{equation*}
provided that (1.1) has the unique two point property on $\I$.

\begin{lemma} \label{Olver-Reid recessive solution}  \cite[Theorems 3.56]{Ahlbrandt-Peterson1996}
Let  $\ld\in \R$. Assume that system $(1.1)$ has  a recessive solution $Y_0=(U_0; V_0)$, which  satisfying $U_0(t_1)$
non-singular and $(1.1)$ has  the unique two point property on $[t_1,\infty)\cap \Z$.
Let $Y(t,s)=(U(t,s);V(t,s))$ be the solution of $(1.1)$,  which
satisfying the boundary conditions
\begin{equation*}
 U(t_1,s)=I_n, \quad  U(s,s)=0
\end{equation*}
Then $Y(t,s)\to Y_0(t)U_0^{-1}(t_1)$ as $s\to \infty$; that is, $Y(t,s)$ converges to
the recessive solution which satisfies the initial condition of $U(t_1)=I_n$.
\end{lemma}

By $B^\dag(t)$ denote   the Moore-Penrose Inverse of the matrix $B$, i.e.,
the unique matrix satisfying $BB^\dag B=B$ and $B^\dag BB^\dag=B^\dag$
such that both $BB^\dag$ and $B^\dag B$ are Hermitian.
It can be easily verified that  $B^\dagger\ge 0$ if $B\ge 0$.
A  solution $y=(u;v)$ of $(1.1)$ is said to have a {\it generalized zero} at $t$  provided that $u(t)=0$ if $t=0$,
and
\begin{align*}
&u(t-1)\neq 0,\quad  u^*(t)\in {\rm Ran}(\tilde{A}(t-1)B(t-1)),\\
&u^*(t-1)B^\dag(t-1)(I_n-A(t-1))u(t)\le 0
\end{align*}
if $t\ge 1$.
System $(1.1)$ is said to be {\it disconjugate} on  $\I$ with respect to  $\ld=\ld_0\in \R$
provided that  every solution $y$ of $(1.1)$ with $\ld=\ld_0$ has at most one generalized zero on $\I$.

For a subinterval $\I_1$ of $\I$, let $\Ld(\I_1)$ denote the vector space of
$n$-dimension vector functions $v(t)$ which are solutions of
$\De v(t)=-A^*(t)v(t)$ and $B(t)v(t)=0$ on $\I_1$. It can be easily verified that
$v(t)\in \Ld(\I_1)$ if and only if
$y(t)=(0; v(t))$ is a solution of (1.1) on $\I_1$ for any $\ld\in\C$.
System $(1.1)$ is said to be {\it normal} on some subinterval $\I_1$ of $\I$ if $\Ld(\I_1)$  is zero-dimensional.
If $(1.1)$ is normal on every  subinterval of $\I$,
then system $(1.1)$ is said to be {\it identically normal} on $\I$.

Further, we denote
\begin{eqnarray*}
&&\mathcal{L}_1(y)(t):=-\De v(t)+C(t) u(t+1)-A^*(t)v(t),\\
&&\mathcal{L}_2(y)(t):=\De u(t)-A(t)u(t+1)-B(t)v(t),\\
&&\mathcal{D}_1(\I):=\{y\in l^2_W(\I):\;\mathcal{L}_1(y)(t)=\ld W(t)u(t+1), \; t\in \I \},\\
&&\mathcal{D}_{2}(\I):=\left\{y\in l^2_W(\I):\; \mathcal{L}_2(y)(t)=0, \; t\in\I\;\right\},\\
&&\mathcal{D}_{0}(\I):=\left\{y\in \mathcal{D}_{2}(\I):\; u(0)=u(t)=0 {\rm \; for \; sufficiently \; large}\; t\right\}.
\end{eqnarray*}
Note that  $\mathcal{D}_2(\I)$ is independent of $\ld$ and it is called the admissive space of (1.1) on $\I$.
A  quadratic form associated with (1.1) is defined by
\begin{equation*}
  \mathcal{F}_{\ld}(y):=\sum_{t\in \I}u^*(t+1)\tilde{C}(t,\ld)u(t+1)
  +\sum_{t\in \I}v^*(t)B(t)v(t),\quad y\in \mathcal{D}_{0}(\I),
\end{equation*}
where $\tilde{C}(t,\ld)=C(t)-\ld W(t)$ by (\ref{ABC}). It is clear that $\mathcal{F}_{\ld}(y)=0$
for any $y\in \mathcal{D}_{2}(\I)$ with $u(t)\equiv 0$ on $\I$.
$\mathcal{F}_{\ld}(\cdot)$  is said to be {\it positive definite} on $\I$ if $\mathcal{F}_{\ld}(y)>0$
for any $y\in \mathcal{D}_{0}(\I)$ with $u(t)\not\equiv 0$ on $\I$.

For convenience, we denote the following assumptions and make some statements on them.
\begin{itemize}
\item[${\bf(A_1)}$] System $(1.1)$  is disconjugate on  $\I$ with respect to  $\ld=\ld_0\in \R$.
\item[${\bf(A_2)}$] System $(1.1)$ is identically normal on $\I$.
\item[${\bf(C_1)}$] $\mathcal{F}_{\ld_0}(\cdot)$  is   positive definite on $\I$.
\item[${\bf(C_2)}$] System $(1.1)$ has  the  unique two point property on  $\I$ with respect to  $\ld=\ld_0\in \R$.
\end{itemize}

\begin{remark} \label{remark}
\begin{enumerate}
  \item It can be easily verified  that $(A_2)$ holds if $B(t)>0$ on $\I$.
But the system $(1.1)$ specified by $(\ref{2nth equation})$ is identically normal on $\I$,
although  $B(t)$ is  singular  for any  $t\in \I$.

  \item  It is clear that $(C_2)$ implies $(A_2)$.
  \item $(A_1)$ and  $(A_2)$ imply  $(C_2)$.
  In fact, Let $y=(u;v)$ be a solution of  $(1.1)$  with $\ld=\ld_0$
satisfying $u(t_1)=u(t_2)=0$  with $t_1<t_2$.
It is clear that $y=(u;v)$   has a  generalized zero in  $[0,t_1]\cap \Z$.
In addition, one can get  that  $u(t)\equiv0$ for  $t\in [t_1, t_2] \cap \Z$.
Otherwise, $y=(u;v)$  would have another   generalized zero in  $[t_1+1,t_2]\cap \Z$.
which  contradicts  the disconjugacy of $(1.1)$.
Further, one get  that $v(t)\equiv0$   for  $t\in [t_1, t_2-1] \cap \Z$.
So, it follows that $y(t_1)=0$, which together with the invertibility of  $I_n-A(t)$, implies that
$y(t)\equiv0$ on $\I$.
\item $(A_1)$ is equivalent to  $(C_1)$. In fact,
 $(A_1)$ holds if and only if $(1.1)$ is disconjugate with respect to  $\ld_0$ on any bounded subinterval of $\I$.
It has been shown by  \cite[Theorem 2]{Bohner1996}
that system $(1.1)$ is disconjugate with respect to $\ld_0$ on a bounded subinterval $\I_1$ of $\I$
if and only if $\mathcal{F}_{\ld_0}(y)$  is positive definite on the corresponding space $\mathcal{D}_{0}(\I_1)$.
\end{enumerate}
\end{remark}

Now we  show the existence of the recessive solutions of (1.1).

\begin{theorem}\label{exist of recessive solution}
Assume that $(A_1)-(A_2)$ hold.
Then for any $t_1\ge 1$, 
system $(\ref{HS1})$ with $\ld=\ld_0$ has a  recessive solution $\tilde{Y}=(\tilde{U};\tilde{V})$
 which satisfying $\tilde{U}(t_1)=I_n$ and  $\tilde{U}(t)$ non-singular for $t\ge t_1$.
\end{theorem}

\begin{proof} Since $(A_1)$ holds,
every solution $y$ of $(1.1)$ with $\ld=\ld_0$ has at most one generalized zero on $\I$.
Let $Y=(U;V)$ be the solution of (\ref{HS1}) with $\ld=\ld_0$
satisfying the initial value condition $Y(0)=(0;I_n)$.
 We claim that  $U(t)$ is non-singular for all $t\ge 1$. Assume not.
 Then there exists $s\ge 1$ and  a non-trivial vector
$\eta\in \C^{n}$  such that $U(s)\eta=0$.
Set $y(t)=(U(t)\eta;V(t)\eta)$. Then $y(t)$ is a non-trivial solution with boundary value
\begin{equation*}
  u(0)=u(s)=0.
\end{equation*}
which yields  $y(t)\equiv0$ on $\I$   by (3) of Remark \ref{remark}.
This contradicts with the initial value condition $Y(0)=(0;I_n)$.
Hence $U(t)$ is non-singular for all $t\ge 1$.

It is clear  that $Y(t)$ is a prepared basis of (\ref{HS1}).
We can get by Lemma \ref{construct conjoined solution} that $U^{-1}(t+1)\tilde{A}(t)B(t)(U^{-1})^*(t)$
is Hermitian for $t\ge 1$. In addition, the assumption $(A_1)$ implies  that
\begin{equation*}
U^*(t)B^{\dagger}(t)(I_n-A(t)) U(t+1)\ge 0,\quad t\ge 1,
\end{equation*}
and consequently
\begin{equation}\label{UABU Hemite}
 U^{-1}(t+1)\tilde{A}(t)B(t)(U^{-1})^*(t)\ge 0,\quad t\ge 1.
\end{equation}
Define
\begin{equation*}
  S(1)=0, \quad S(t)=\sum_{s=1}^{t-1}U^{-1}(s+1)\tilde{A}(s)B(s)(U^{-1})^*(s),\quad t\ge 2,
\end{equation*}
and set
\begin{eqnarray}\label{U(t)}
\left\{\begin{array}{ll}
         U_1(t)=U(t)(I_n+S(t)), &\\
        V_1(t)=V(t)(I_n+S(t))+(U^{-1})^*(t),&t\ge 1.
       \end{array}\right.
\end{eqnarray}
Again by  Lemma \ref{construct conjoined solution}, $Y_1=(U_1;V_1)$ is a prepared basis solution of (\ref{HS1}) with $\ld=\ld_0$.
It follows from   (\ref{UABU Hemite})  that
\begin{equation*}
  0=S(1)\le S(t)\le S(t+1),\quad t\ge 2.
\end{equation*}
This implies that  $U_1(t)$ is non-singular for $t\ge 1$.
Again by Lemma  \ref{construct conjoined solution}, $Y=(U;V)$ can be written as
\begin{eqnarray} \label{U2(t)}
      \left\{\begin{array}{ll}
                U(t)=U_1(t)(I_n-S_1(t)), \\
        V(t)=V_1(t)(I_n-S_1(t))-(U_1^{-1})^*(t),& t\ge 1,
             \end{array}\right.
\end{eqnarray}
where
\begin{align*}
 S_1(1)=0,\quad  S_1(t)=\sum_{s=1}^{t-1}U_1^{-1}(s+1)\tilde{A}(s)B(s)(U_1^{-1})^*(s), \quad t\ge 2.
\end{align*}

 It follows from the first formulae of (\ref{U(t)}) and (\ref{U2(t)}) that
\begin{align*}
I_n=(I_n+S(t))(I_n-S_1(t)), \quad t\ge 1.
\end{align*}
Since $S(t)\ge 0$ and it is  non-decreasing,  one has that $I_n-S_1(t)>0$ and it is non-increasing.
Therefore, the limit
\begin{equation*}
  S_1(\infty)=\lim_{t\to \infty}S_1(t)
\end{equation*}
exists and it satisfies
\begin{equation*}
 0< S_1(\infty)< I_n.
\end{equation*}
Thus, $Y_1=(U_1;V_1)$ is a dominant solution with $U_1(t)$  non-singular for $t\ge 1$.

Define
\begin{equation*}
 S_2(t)=\sum_{s=t}^{\infty}U_1^{-1}(s+1)\tilde{A}(s)B(s)(U_1^{-1})^*(s),\quad t\ge 1.
\end{equation*}
It can been shown similarly as above that $ 0< S_2(t)< I_n$ for all $t\ge 1$.
Set
\begin{align*}
\left\{\begin{array}{ll}
         U_2(t)=U_1(t)S_2(t), &\\
    V_2(t)=V_1(t)S_2(t)-(U_1^{-1})^*(t),&   t\ge 1.
       \end{array}\right.
\end{align*}
Then $Y_2=(U_2;V_2)$  is determined uniquely on $\I$ and it follows from
Lemma \ref{dominant to recessive solution} that
$Y_2(t)$ is a recessive solution of (1.1).
It is clear that  $U_2(t)$ is invertible for $t\ge 1$.
Set $\tilde{Y}(t)=Y_2(t)U_2^{-1}(t_1)$. Then $\tilde{Y}(t)$ is still a recessive solution satisfying
$\tilde{U}(t_1)=I_n$ and $\tilde{U}(t)$ is invertible for $t\ge t_1$.
The proof is complete.
\end{proof}


\subsection{The  existence  of the Friedrichs extension}

Denote
\begin{equation*}
l^2_{W}(\I):=\left\{y\in l(\I):\;\sum_{t\in
\I}u^*(t+1)W(t)u(t+1)<\infty\right\}
\end{equation*}
with the semi-scalar product
\begin{eqnarray*}
\langle y_1,y_2\rangle_W:=\sum_{t\in \I}u_2^*(t+1)W(t)u_1(t+1).
\end{eqnarray*}
Further,  we define $\|y\|_W:=\langle y, y\rangle_W^{1/2}$ for $y\in l^2_{W}(\I)$,
and  introduce the following quotient space
\begin{equation*}
L^2_{W}(\I):=l^2_W(\I)/\{y\in l^2_W(\I):\;
\|y\|_W=0\}.
\end{equation*}
Then by  \cite[Lemma 2.5]{Shi2006}, $L^2_W(\I)$ is a Hilbert space with
the inner product $\langle \cdot,\cdot\rangle_W$. For any $y\in
l^2_W(\I)$, we denote by $[y]$   the corresponding
equivalent class in $L^2_W(\I)$. Denote
\begin{eqnarray*}
l^2_{W,0}(\I):=\{y\in l^2_{W}(\I):\;\exists\; s\in \I \;{\rm s.t.}\;y(0)=y(t)=0,\; t\ge s\}.
\end{eqnarray*}

Consider the following non-homogeneous system corresponding to (1.1)
\begin{eqnarray}\label{NHS}
\left\{\begin{array}{l}
        \De u(t)=A(t)u(t+1)+B(t)v(t),\\
        \De v(t)=C(t)u(t+1)-A^*(t)v(t)- W(t)u'(t+1),\quad t\in \I,
       \end{array}\right.
\end{eqnarray}
and  denote
\begin{eqnarray*}
&&H:=\{([y],[y'])\in (L^2_W(\I))^2: \;\exists\;y\in [y]\;s.t.\;(\ref{NHS})\; {\rm holds}\}, \\
&&H_{00}:=\{([y],[y'])\in H:  \exists \;y\in [y] \; {\rm s.t.\;} y\in l^2_{W,0}(\I)\;{\rm and }\;(\ref{NHS})\; {\rm holds}\},\nonumber\\
&&H_0:=\overline{H}_{00},
\end{eqnarray*}
where $H$, $H_{00}$, and $H_0$  are called the maximal, pre-minimal,
and minimal subspaces  corresponding to system $(1.1)$, separately.

It is clear that $H_{00}$ and $H_{00}$ are Hermitian. It follows from \cite[Theorem 3.1]{Ren2011}
that $H_{00}^*=H_0^*=H$ and from \cite[Theorem 3.2]{Ren2014Laa} that
\begin{align}
H_{0}=\{(y,[y'])\in H: y(0)=0, (x,y)(\infty)=0 \; {\rm for
\; all}\; x\in { \mathcal{D}}(H)\}.
\end{align}

We say that  the {\it definiteness condition} for system $(1.1)$ holds on $\I$,
if there exists a finite subinterval $\I_0=[t_0,s_0]\cap\Z\subset \I$ such that for any
$\ld\in\C$,  every non-trivial solution $y(t)$ of $(1.1)$ satisfies
\begin{align*}
\sum_{t\in \I_0}u^*(t+1)W(t)u(t+1)>0.
\end{align*}
For convenience, we denote
\begin{itemize}
\item[${\bf(A_3)}$]  The  definiteness condition for system $(1.1)$ holds on   $\I$.
\end{itemize}

\begin{remark}\label{remark on A_3} Assumption $(A_3)$  is important and necessary to
the characterization of Fridrichs extension of the minimal subspace generated by $(1.1)$.
\begin{enumerate}

\item It has been shown by \cite[Corollary 5.1]{Ren2011} that the positive and negative deficiency indices
are equal to the the number of linearly independent square summable solutions of $(1.1)$ with $\ld$ in the upper or lower half-planes,
if and only if the  definiteness condition for system $(1.1)$ holds.  With the assumption $(A_3)$,
a complete characterization of all the self-adjoint  extensions for $(\ref{HS2})$
 has been obtained in
terms of boundary conditions via linearly independent square summable
solutions  \cite{Ren2014Laa}.

\item It has been shown in \cite[Theorem 4.2]{Ren2011}  that   $(A_3)$ holds
if and only if    for any $([y],[y'])\in H$, there
exists a unique $y\in [y]$ such that $(\ref{NHS})$ holds.
Therefore, we can write briefly $(y,[y'])\in H$  instead of $([y],[y'])\in H$.

\item If  $(A_3)$  holds, then any two elements of $H$ can be patched up to
construct another new element of $H$. For example, for any $ (x,[x']), (y,[y'])\in H$, there exists $(z,[z'])\in H$ satisfying
\begin{align*}
z(t)=x(t) \;\;{\rm for}\;\; t\leq t_0,\quad  z(t)=y(t)\;\;{\rm
for}\; \;t \ge s_0+1.
\end{align*}
In addition, there exist  $(z_j,[z_j'])\in H$ satisfying
\begin{align*}
z_j(t_0)=e_j,\quad  z_j(t)=0, \quad t \ge s_0+1,\quad 1\le j\le 2n,
\end{align*}
where
\begin{align*}
    e_i:=(\underbrace{0,\ldots,0}_{i-1}, 1,0,\ldots,0)^T\in \C^{2n}.
\end{align*}
For details, see \cite[Lemma 3.3, Remark 3.3]{Ren2014Laa}.

\item The definiteness condition for $(1.1)$ is independent of $\ld$,
and some sufficient conditions  have been given  \cite[Section 4]{Ren2011}.
\end{enumerate}
\end{remark}

\begin{lemma}\label{inner charact} Assume that  $(A_3)$  holds.
For any $(y_i,[y_i'])\in H$ with $y_i=(u_i;v_i)$ and $y_i'=(u_i';v_i')$, $i=1,2$, it follows
\begin{align}
  \langle y_1',y_2\rangle_W&=\sum_{t\in \I}u_2^*(t+1)\mathcal{L}_1(y_1)(t)\nonumber \\
  &=\sum_{t\in \I}\{u_2^*(t+1)C(t)u_1(t+1)+v_2^*(t)B(t)v_1(t)\}-u_2^*(t)v_1(t)|_0^{\infty}.\label{inner product}
\end{align}
In particularly, for any $(y,[y'])\in H$,
\begin{align}
  \langle y',y\rangle_W=\sum_{t\in \I}\{u^*(t+1)C(t)u(t+1)+v^*(t)B(t)v(t)\}-u^*(t)v(t)|_0^{\infty},
\end{align}
and  for any $(y,[y'])\in H_{00}$,
\begin{align}
  \langle y',y\rangle_W =\sum_{t\in \I}u^*(t+1)C(t)u(t+1)+\sum_{t\in \I}v^*(t)B(t)v(t).
\end{align}
\end{lemma}

\begin{proof} It suffices to show (\ref{inner product}) holds.
Let $(y_i,[y_i'])\in H$ with $y_i=(u_i;v_i)$ and $y_i'=(u_i';v_i')$, $i=1,2$.
Then it follows from (\ref{NHS}) that
\begin{eqnarray}\label{inner product2}
\left\{\begin{array}{l}
         \De u_i(t)=A(t)u_i(t+1)+B(t)v_i(t), \\
         \De v_i(t)=C(t)u_i(t+1)-A^*(t)v_i(t)- W(t)u'_i(t+1),\quad t\in \I.
       \end{array}
\right.
\end{eqnarray}
By the second formula of (\ref{inner product2}), one can get
\begin{eqnarray*}
W(t)u'_1(t+1)=C(t)u_1(t+1)- [\De v_1(t)+A^*(t)v_1(t)]=\mathcal{L}_1(y_1)(t),
\end{eqnarray*}
and further one can get
\begin{align}\label{diff1}
\langle y_1',y_2\rangle_W &=\sum_{t\in \I}u_2^*(t+1)W(t)u_1'(t+1)=\sum_{t\in \I}u_2^*(t+1)\mathcal{L}_1(y_1)(t)\nonumber\\
  &=\sum_{t\in \I}\left\{u_2^*(t+1)C(t)u_1(t+1)-u_2^*(t+1)[\De v_1(t)+A^*(t)v_1(t)]\right\}.
\end{align}
By using the formula
\begin{equation*}
  \Delta (u^*(t)v(t))=\De u^*(t) v(t)+u^*(t+1)\De v(t)
\end{equation*}
and the first formula of (\ref{inner product2}) one can get  that
\begin{align}\label{diff2}
-u_2^*(t+1)[\De v_1(t)+A^*(t)v_1(t)]=v_2^*(t)B(t)v_1(t)-\Delta (u_2^*(t)v_1(t)).
\end{align}
Inserting (\ref{diff2})  into (\ref{diff1}), one can obtain the second equation of  (\ref{inner product}). The proof is complete.
\end{proof}

Let $\ld\in \C$ and $(y,[y'])\in H-\ld I$.
Since $I_n-A(t)$ is invertible, one can get from (\ref{NHS}) that
$y(t)$ is  determined uniquely by $[y']$ and the initial value $y(0)$.
In particular, $y$ is  determined uniquely by $[y']$ for any $(y,[y'])\in H_{00}-\ld I$.
This means that $(H_{00}-\ld I)^{-1}$ is an Hermitian operator.
Now,  we introduce a  sesquilinear form $\bf t_{\ld}$ associated with $ (H_{00}-\ld I)^{-1}$:
\begin{align*}
{\bf t}_{\ld}[y']:&=\langle y',y\rangle_W-\ld \langle y,y\rangle_W,\quad \forall\; (y,[y'])\in H_{00}-\ld I.
\end{align*}
It is clear  $\mathcal{D}({\bf t_{\ld}})=\mathcal{R}( H_{00}-\ld I)=\mathcal{D}( (H_{00}-\ld I)^{-1})$.
One get  $\bf t_{\ld}$ is closable by Lemma \ref{Hermite relation to form}.
In addition, it follows from Lemma   \ref{inner charact}  that for any $(y,[y'])\in H_{00}-\ld I$,
\begin{align*}
{\bf t}_{\ld}[y']&=\mathcal{F}_{\ld}(y)=\sum_{t\in \I}u^*(t+1)\tilde{C}(t,\ld)u(t+1)+\sum_{t\in \I}v^*(t)B(t)v(t).
\end{align*}

\begin{theorem}\label{bounded below}
Assume that $(A_1)-(A_3)$ hold and let $\ld<\ld_0$. Then $H_{0}-\ld  I$  is bounded from below with the lower bound $\ga>0$.
And consequently, $d_+(H_0)=d_-(H_0):=d$ and therefore  $H_{0}$ has  a unique self-adjoint extension (Friedrichs extension)
with the  lower bound $\ga+\ld$.
\end{theorem}

\begin{proof} It suffices to show  $H_{00}-\ld I$  is bounded from below with the lower bound $\ga>0$.
Suppose that $H_{00}-\ld I$  is not bounded from below with the lower bound $\ga>0$.
Then for any given $\varepsilon>0$, there exists $(y,[y'])\in H_{00}$ such that
\begin{align*}
 \langle y'-\ld y,y\rangle_W\le \varepsilon \langle y,y\rangle_W.
\end{align*}
In particular, by taking $\varepsilon=(\ld_0-\ld)/2$, we obtain
\begin{align*}
\mathcal{F}_{\ld_0}(y)= \langle y'-\ld_0y, y \rangle_W<0.
\end{align*}
This   contradicts with the assumption $(A_1)$ by (4) of Remark \ref{remark}.
Hence, $H_{00}-\ld I$ and consequently $H_{0}-\ld I$ is  bounded from below with the lower bound $\ga>0$.
This yields that $H_{0}$ is  bounded from below with lower bound $\ga+\ld$ and
it   follows from Lemma \ref{lower bound def index equal} that
$d_+(H_0)=d_-(H_0)$ and consequently  $H_{0}$ has self-adjoint extensions.

The boundedness of $H_0-\ld I$ implies the boundedness of ${\bf t}_{\ld}[\cdot]$,  as well as
its closure ${\bf \bar{t}}_{\ld}[\cdot]$.
Further, it follows from Lemma \ref{existence F extension} that  there exists a unique self-adjoint  extension of $H_0-\ld I$
with the same lower bound. This means that  $H_{0}$ has a unique Friedrichs extension with lower bound $\ga+\ld$.
The proof is complete.
\end{proof}

In the following, we denote the Friedrichs extension of $H_0$ by $H_F$.
Then $H_F-\ld I$ is the Friedrichs  extension of $H_0-\ld I$.

\begin{lemma}\label{hf u(0)=0}
 Assume that  $(A_1)-(A_3)$ hold. Then  $y(0)=0$ for any $(y,[y'])\in H_F$.
\end{lemma}

\begin{proof} Since $\mathcal{D}(H_{F})=\mathcal{D}(H_{F}-\ld I)$ for all $\ld\in \R$, it suffices to show
 $y(0)=0$ for any $(y,[y'])\in H_F-\ld I$ with some $\ld<\ld_0$.

Let $\ld<\ld_0$ and  $(y,[y'])\in H_{F}-\ld I$. Note that $H_{F}-\ld I$ is a Friedrichs extension
of $H_{00}-\ld I$, and   $(H_{F}-\ld I)^{-1}$ is a Friedrichs extension
of $(H_{00}-\ld I)^{-1}$.
It follows from \ref{existence F extension} that $\mathcal{D}((H_{F}-\ld I)^{-1})\subset \mathcal{D}(\bar{\bf t})$.
So, there exists  $(y_m,[y'_m])\in H_{00}-\ld I$ such that $[y'_{m}]\stackrel{\bf t_{\ld}}{\longrightarrow}\;[y']$.
That is,
\begin{equation}\label{u'-um'}
\lim_{m\to \infty}\sum_{t\in\I}(u'-u'_m)^*(t+1)W(t)(u'-u'_m)(t+1)=0,
\end{equation}
and
\begin{align*}
{\bf \bar{t}}_{\ld}[y']&=\lim_{m\to \infty}{\bf t}_{\ld}[y'_m]\\
&=\lim_{m\to \infty}\sum_{t\in\I}u_m^*(t+1)\tilde{C}(t,\ld)u_m(t+1)+\sum_{t\in\I}v_m^*(t)B(t)v_m(t).
\end{align*}
It is clear that (\ref{u'-um'}) yields
\begin{equation}\label{wu'm}
\lim_{m\to \infty}W(t)u'_m(t+1)=W(t)u'(t+1),\quad t\in \I.
\end{equation}

Since $(y_m,[y_m'])\in H_{00}-\ld I$, it follows that
\begin{eqnarray}\label{inner product m}
\left\{\begin{array}{l}
         \De u_m(t)=A(t)u_m(t+1)+B(t)v_m(t), \quad t\in \I,\\
         \De v_m(t)=\tilde{C}(t,\ld)u_m(t+1)-A^*(t)v_m(t)- W(t)u'_m(t+1).
       \end{array}
\right.
\end{eqnarray}
Since $u_m(0)=v_m(0)=0$ for all $m$, it follows from the first formula of  (\ref{inner product m}) that $u_m(1)=0$ for all $m$
and then it follows from the second formula that (\ref{inner product m}) that $v_m(1)=- W(0)u'_m(1)$.
This yields
\begin{equation*}
  \lim_{m\to \infty} v_m(1)=- W(0)u'(1)
\end{equation*}
by using (\ref{wu'm}).
Inserting  $u_m(1)=0$ and $v_m(1)=- W(0)u'_m(1)$ into the first formula of (\ref{inner product m}) and letting $m\to \infty$,
one get that
\begin{equation*}
  \lim_{m\to \infty}u_m(2)=(I_n-A(1))^{-1}B(1)W(0)u'(1).
\end{equation*}
Repeating the above procession, one can get that for each $t\in \I$,  $\lim_{m\to \infty}y_m(t)$ exists, which is
denoted by $y_0(t)$.
It is clear  $(y_0,[y'])\in H_F-\ld I$ and $y_0(0)=0$.
Since $H_F-\ld I$  is bounded with lower bound $\gamma>0$, $(H_{F}-\ld I)^{-1}$ is a bounded operator.
Therefore,  $y_0(t)=y(t)$ for all $t\in \I$. In particular, $y(0)=y_0(0)=0$. The proof is complete.
\end{proof}

\subsection{The characterization  of the recessive solutions}

Let  $\tilde{Y}_s=(\tilde{U}_s; \tilde{V}_s)$ be the solution of $(1.1)$ on $\I$,
which satisfies the boundary conditions
\begin{equation*}
 \tilde{U}_s(1)=I_n, \quad  \tilde{U}_s(s)=0, \quad (s\ge s_0+1)
\end{equation*}
and let $\tilde{Y}=(\tilde{U}; \tilde{V})$
be the recessive solution  of (1.1), satisfying $\tilde{U}(1)=I_n$.
Denote
\begin{equation*}
  \tilde{Y}(t)=[\tilde{y}_1(t), \tilde{y}_2(t), \ldots, \tilde{y}_n(t)],\quad
  \tilde{Y}_s(t)=[\tilde{y}_{1s}(t), \tilde{y}_{2s}(t), \ldots, \tilde{y}_{ns}(t)].
\end{equation*}
It  follows from Lemma \ref{Olver-Reid recessive solution} that
\begin{equation*}
  \tilde{y}_j(t)=\lim_{s\to \infty}\tilde{y}_{js}(t),\quad t\in \I.
\end{equation*}

It follows from (3) of Remark \ref{remark on A_3} that
there exist  $(z_j,[z_j'])\in H$ satisfying
\begin{align}\label{zj=ej}
z_j(0)=\tilde{y}_j(0),\quad  z_j(t)=0, \quad t \ge s_0+1,\quad 1\le j\le n.
\end{align}
Define
\begin{align*}
  y_{j}:=\tilde{y}_j-z_j,\quad 1\le j\le n.
 \end{align*}
It is clear that $y_j(0)=0$, and
\begin{equation*}\label{H-ldI}
 (y_j,[\ld z_j-z_j'])\in H-\ld I,
\end{equation*}

On the other hand,  let $\ld<\ld_0$. Then $(H_{F}-\ld I)^{-1}$ is a bounded operator defined on $L_W^2(\I)$.
Hence, there exists a unique $\check{y}_j\in l^2_W(\I)$ such that
\begin{equation}\label{H-ldI2}
 (\check{y}_j,[\ld z_j-z_j'])\in H_F-\ld I.
\end{equation}
By noting $y_j(0)=0$ and by using Lemma \ref{hf u(0)=0}, one get that
$y_j(t)=\check{y}_j(t)$ on $\I$.

Based on the above discussion, we get the  following  characterization of the recessive  solutions.

\begin{theorem}\label{cha of recessive solu}
Assume that  $(A_1)-(A_3)$ hold and  $\ld<\ld_0$. Let $ \tilde{Y}(t)=[\tilde{y}_1, \tilde{y}_2, \ldots, \tilde{y}_n](t)$
be the recessive  solution of $(1.1)$ which satisfying $\tilde{U}(1)=I_n$.
Then
\begin{align}\label{recessive solution charac}
   \tilde{y}_j(t)=z_j(t)+\check{y}_j(t), \quad t\in \I,\quad j=1,2,\ldots,n,
\end{align}
where $z_j$ is defined by $(\ref{zj=ej})$ and $\check{y}_j$ is defined by $(\ref{H-ldI2})$.
\end{theorem}

\section{Characterizations of Fridriches extensions}

In this section,   the most simple characterization of the matrix $\Tht$  is established
by using of the  recessive  solutions.
Based on this, we  establish a characterization of the Fridriches extension $H_F$ in
terms of boundary conditions via linear independent recessive  solutions.

We assume that $(A_1)-(A_3)$ hold and fix $\ld<\ld_0$.
It follows from Theorem \ref{bounded below} and Lemma \ref{lower bound def index equal}
that  the minimal subspace $H_0$ is bounded from below
and $d_+(H_0)=d_-(H_0)$.
Denote
\begin{equation*}
  d_+(H_0)=d_-(H_0):=d
\end{equation*}
Then  $n\le d\le 2n$, and it follows from (1) of Remark \ref{remark on A_3} that
$(1.1)$  has $d$
linearly independent  solutions $\tht_1,\tht_2,\ldots,\tht_d$ in $l^2_W(\I)$.
As it has discussed in \cite[Section 4]{Ren2014Laa},  $\tht_1,\tht_2,\ldots,\tht_d$ can be arranged such that
\begin{align}\label{tht}
\Tht:=\left((\tht_i,\tht_j)(t)\right)_{1\le i,j\le
2d-2n}
\end{align}
is a constant matrix and it is invertible.

\begin{lemma}\label{sse}\cite[Theorem 5.8]{Ren2014Laa}
Assume that  $(A_1)-(A_3)$ hold.  A subspace
$T\subset(L^2_W(\I))^2$ is a self-adjoint extension  of $H_{0}$ if and only if there
exist two matrices $M_{d\times 2n}$ and $N_{d\times (2d-2n)}$ such
that
\begin{align}\label{MN condition selfadjoint}
    {\rm rank}\, (M,N)=d,\quad MJM^*-N\Theta^TN^*=0,
\end{align}
and
\begin{align}\label{T SSE}
    T=\{(y,[y'])\in H:
My(0)-N\left(\ba{c}[y,\tht_1](\infty)\\
\vdots \\
\left[y,\tht_{2d-2n}\right](\infty)\ea\right)=0\}.
\end{align}
where $\Theta$ is defined by $(\ref{tht})$.
\end{lemma}

\begin{remark}
 If matrices $M$ and $N$ satisfy the conditions in Lemma $\ref{sse}$, we call $(M,N)$ a self-adjoint  boundary condition.
  Let  $(M,N)$ and $(M',N')$ be two self-adjoint  boundary conditions.
If there exists a $d\times d$ invertible matrix $G$ such that
\begin{equation*}
M'=GM, \quad N'=GN,
\end{equation*}
then $(M,N)$ and $(M',N')$ are regarded as  one self-adjoint  boundary condition.
\end{remark}

\begin{lemma}\label{Tht lemma}
Assume that  $(A_1)-(A_3)$ hold and let $\ld<\ld_0$.
Then $(1.1)$ has $d$ linearly independent solutions
$\tilde{y}_1(t), \tilde{y}_2(t), \ldots, \tilde{y}_n(t)$, $\hat{y}_1(t), \hat{y}_2(t), \ldots, \hat{y}_{d-n}(t)$  in $l^2_W(\I)$ such that
\begin{enumerate}
  \item $\tilde{Y}(t)=[\tilde{y}_1(t), \tilde{y}_2(t), \ldots, \tilde{y}_n(t)]$ is a recessive solution of $(1.1)$ satisfying $\tilde{U}(1)=I_n$;
  \item $\hat{Y}(t)=[\hat{y}_{1}(t),\hat{y}_{2}(t), \ldots, \hat{y}_{d-n}(t)]$ satisfies $\hat{U}(1)=0$ and  the  $i_1,\ldots,i_{d-n}$-th
  rows of $\hat{V}(1)$ form $I_{d-n}$.
\item Let
\begin{equation}\label{tht1thtd}
 \tht_j=\tilde{y}_{i_j}, \quad \tht_{d-n+j}=\hat{y}_{j},\quad j=1,2,\ldots,d-n.
\end{equation}
Then $\Tht$ defined by ($\ref{tht}$) has the form
\begin{equation}\label{Tht1}
  \Tht=\left(
         \begin{array}{cc}
           0 & I_{d-n} \\
           -I_{d-n} & 0 \\
         \end{array}
       \right).
\end{equation}
\end{enumerate}
\end{lemma}

\begin{proof}
It follows from  Theorem \ref{exist of recessive solution} that the recessive solution of $(1.1)$ exists,
which is denoted by $\tilde{Y}(t)$ and satisfies $\tilde{U}(1)=I_n$.
By the definition of the recessive solution, one get that
the $n$  columns of $\tilde{Y}(t)$, denote by $\tilde{y}_1(t), \tilde{y}_2(t), \ldots, \tilde{y}_n(t)$,
are $n$  linear independent solutions belonging to $l_W^2$.

By $y_{n+1}(t), y_{n+2}(t), \ldots, y_d(t)$ denote the other $d-n$ linear independent solutions of (1.1) in $l_W^2(\I)$,
and denote $Y_1(t)=[y_{n+1}(t), y_{n+2}(t), \ldots, y_d(t)]$.
Then $Y_1(t)$ is a $2n\times (d-n)$ matrix-valued  solution of (1.1).
It is clear  $\hat{Y}(t):=Y_1(t)-\tilde{Y}(t)U_1(1)$ is still a $2n\times (d-n)$ matrix-valued solution of (1.1),
satisfying  $\hat{U}(1)=0$. Denote
\begin{equation*}
 \hat{Y}(t)=[\hat{y}_{1}(t),\hat{y}_{2}(t), \ldots, \hat{y}_{d-n}(t)].
\end{equation*}
Then  $\tilde{y}_1(t), \tilde{y}_2(t), \ldots, \tilde{y}_n(t), \hat{y}_{1}(t),\hat{y}_{2}(t), \ldots, \hat{y}_{d-n}(t)$
are $d$ linear independent solutions of (1.1) in $l_W^2$.

It  follows from $\hat{U}(1)=0$  that $\hat{V}(1)$ has rank $d-n$.
We first consider the case that the first $d-n$ rows of $\hat{V}(1)$ form a $(d-n)\times (d-n)$ non-singular sub-matrix.
By multiplying a   $(d-n)\times (d-n)$ non-singular constant matrix from right side if necessary,
we can get a new $2n\times (d-n)$ matrix-valued solution of (1.1),  still denoted by $\hat{Y}(t)$,
whose the first  $d-n$ rows of $\hat{V}(1)$ form $I_{d-n}$. In this case,
\begin{equation*}
  (\tilde{Y}(1),\hat{Y}(1))=\left(
                                  \begin{array}{cc}
                                    I_n & O_{n\times (d-n)} \\
                                    \tilde{V}(1) & \begin{array}{c}
                                                       I_{d-n} \\
                                                       \hat{V}_{1}(1)
                                                     \end{array}
                                     \\
                                  \end{array}
                                \right),
\end{equation*}
where $\hat{V}_{1}(1)$  is a $(2n-d)\times (d-n)$ matrix.

For the general case that the $i_1,\ldots,i_{d-n}$-th   rows of $\hat{V}(t_1)$ form a $(d-n)\times (d-n)$ non-singular sub-matrix.
By multiplying a   $(d-n)\times (d-n)$ non-singular constant matrix from right side if necessary,
we can get a new $2n\times (d-n)$ matrix solution of (1.1),  still denoted by $\hat{Y}(t)$,
whose  the  $i_1,\ldots,i_{d-n}$-th    rows of $\hat{V}(t_1)$ form $I_{d-n}$.

Let $\tht_j$ be defined as (\ref{tht1thtd}).
Then it can be verifies directly that
\begin{align*}
 & \left((\tht_i,\tht_j)(1)\right)_{1\le i\le d-n,d-n+1\le j\le 2d-2n}\\
  =&-\left((\tht_i,\tht_j)(1)\right)_{d-n+1\le i\le 2d-2n,  1\le j\le d-n}=I_{d-n}
\end{align*}
and
\begin{equation*}
\left((\tht_i,\tht_j)(1)\right)_{d-n+1\le i\le 2d-2n,  d-n+1\le j\le 2d-2n}=0.
\end{equation*}
Further,  it follows
from  (\ref{recessive solution charac}) and (\ref{zj=ej}) that
\begin{equation*}
\tilde{y}_j^*J\tilde{y}_k(\infty)=\check{y}_j^*J\check{y}_k(\infty)=\check{y}_j^*(0)J\check{y}_k(0)=0,\quad 1\le j,k\le d-n,
\end{equation*}
where  the self-adjointness of $H_F$ and  $\check{u}_j(0)=0$ $(j=1,2,\ldots,n)$ are used.
This implies that
\begin{equation*}
\left((\tht_i,\tht_j)(1)\right)_{1\le i,j\le d-n}=0.
\end{equation*}
So, (\ref{Tht1}) holds. The whole proof is complete.
\end{proof}

\begin{theorem}
Assume that $(A_1)-(A_3)$ hold. Let $\tht_j$, $j=1,\ldots,d-n$,  be defined by Lemma \ref{Tht lemma}. Then
\begin{align}\label{T1}
    H_F=\{(y,[y'])\in H:
u(0)=0,\quad (y,\tht_{j})(\infty)=0,\quad j=1,2,\ldots, d-n\},
\end{align}
\end{theorem}

\begin{proof}
Denote
\begin{align}\label{T}
    T=\{(y,[y'])\in H:
u(0)=0,\quad (y,\tht_{j})(\infty)=0,\quad j=1,2,\ldots, d-n\}.
\end{align}
We first show  $H_F\subset T$.
For any $(y,[y'])\in H_F$, it follows from Lemma \ref{hf u(0)=0} that $u(0)=0$.
In addition, by the self-adjointness of $H_F$, one has
\begin{equation*}
  (y,\tht_j)(\infty)=(y,\check{y}_{i_j})(\infty)=(y,\check{y}_{i_j})(0)=0,\quad j=1,2,\ldots, d-n,
\end{equation*}
where  (\ref{recessive solution charac}), (\ref{zj=ej}) and $u(0)=\check{u}_{i_j}(0)=0$ are used.
So,  $H_F\subset T$ holds.

Next, we show  $T$  is a self-adjoint extension of $H_{0}$.
It follows from Lemma \ref{Tht lemma} that ($\ref{Tht1}$) holds.
Set
\begin{equation*}\label{M N 2}
  M=\left(
        \begin{array}{cc}
          I_n & 0 \\
          0&0\\
        \end{array}
      \right)_{d\times 2n},\\
      \quad N=\left(
                          \begin{array}{cc}
                          0&0\\
                            I_{d-n} &0 \\
                          \end{array}
                        \right)_{d\times (2d-2n)},
\end{equation*}
It can be easily verified that $(M, N)$ is a self-adjoint boundary condition and  for each $y\in \mathcal{D}(H)$, it follows that
\begin{align*}
My(0)=\left(\ba{c}u(0)\\
0\ea\right),
\end{align*}
and
\begin{align*}
N\left(\ba{c}(y,\tht_1)(\infty)\\\vdots\\
(y,\tht_{2d-2n})(\infty)\ea\right)
=\left(\ba{c}0\\\vdots\\0\\(y,\tht_{1})(\infty)\\\vdots\\
(y,\tht_{d-n})(\infty)\ea\right).
\end{align*}
Therefore,  $T$    is a self-adjoint extension  of $H_{0}$.
This, together with $H_F\subset T$, yields that  $T=H_F$.
The proof is complete.
\end{proof}

\section*{Acknowledgement}
This work is supported by the NSF of Shandong Province, P.R. China [grant numbers  ZR2020MA012 and ZR2021MA070].


\end{CJK*}

\end{document}